\documentclass[11pt]{amsart}
\usepackage[T1]{fontenc}
\usepackage{lmodern}
\usepackage{amsmath,amssymb,amsthm}
\usepackage{needspace}
\usepackage[margin=30mm]{geometry}
\usepackage[hidelinks]{hyperref}
\newtheorem{theorem}{Theorem}[section]
\newtheorem{lemma}[theorem]{Lemma}
\newtheorem{proposition}[theorem]{Proposition}
\newtheorem{corollary}[theorem]{Corollary}
\theoremstyle{definition}
\newtheorem{example}[theorem]{Example}
\theoremstyle{remark}
\newtheorem{remark}[theorem]{Remark}
\newcommand{\id}{\mathrm{id}}
\newcommand{\norm}[1]{\left\|#1\right\|}
\newcommand{\cond}[1]{\textup{(#1)}}
\title[Surjective Fischer--Musz\'ely maps]{Surjective Fischer--Musz\'ely maps on positive cones of JB-algebras}
\author{Osamu Hatori}
\address{Niigata University,
Niigata 950-2181, Japan}
\email{oppekepenguin@gmail.com}

\author{Shiho Oi}
\address{Department of Mathematics, Faculty of Science, Niigata University,
Niigata 950-2181, Japan}
\email{shiho-oi@math.sc.niigata-u.ac.jp}

\hypersetup{
  pdfauthor={Osamu Hatori and Shiho Oi},
  pdftitle={Surjective Fischer--Muszely Maps on Positive Cones of JB-Algebras}
}

\date{Working draft, September 16, 2026}
\begin{document}
\begin{abstract}
We prove that every surjective Fischer--Musz\'ely map between positive
cones of arbitrary JB-algebras is additive and positively homogeneous,
and extends uniquely to a bounded positive linear surjection.
In particular, this gives a complete affirmative answer to Moln\'ar's
problem for arbitrary $C^*$-algebras. The theorem includes nonunital
and exceptional JB-algebras and assumes neither injectivity nor
continuity. The proof connects the norm-valued functional equation
to positive-cone rigidity through translation invariance of an
induced pseudometric. For a surjective map satisfying the
Fischer--Musz\'ely identity with uniform error $\varepsilon$, we also
construct a unique bounded positive linear map $L$ at uniform distance
at most $3\varepsilon$. The set $L(A_+)$ is dense in $B_+$,
although $L$ need not be surjective, even when the original map is
continuous and bijective. For bijective FM maps, we obtain weighted
Jordan representations on the biduals. We also extend the
characterization of surjective norm-sum preservers to JB-algebras
and characterize norm arithmetic-midpoint and harmonic-mean identities
on positive invertible cones of unital JB-algebras.
\end{abstract}
\maketitle

\section{Introduction and the main theorem}

The Fischer--Musz\'ely equation is a norm-valued counterpart of the
Cauchy functional equation; see \cite{FM}. For the positive cones of
JB-algebras, we call a map $T:A_+\to B_+$ an \emph{FM-map} if
\begin{equation}\label{eq:FM}
 \norm{T(a+b)}=\norm{T(a)+T(b)}\qquad(a,b\in A_+).
\end{equation}
Unless explicitly stated otherwise, JB-algebras in this paper are
not assumed to be unital. Linear maps between JB-algebras are understood
to be real-linear. In particular, a Jordan isomorphism means a real-linear
bijection preserving the Jordan product.

Moln\'ar asked whether every surjective FM-map between the positive
cones of $C^*$-algebras must be additive; the question is recorded in
\cite[Problem~1.1]{HH}.

We prove additivity for every surjective FM-map between positive
cones of arbitrary JB-algebras. This includes nonunital and exceptional
algebras and settles Moln\'ar's problem for arbitrary $C^*$-algebras.
No injectivity or continuity assumption is imposed on the map.
The resulting characterization is as follows.

\begin{theorem}\label{thm:main}
Let $A,B$ be JB-algebras and let $T:A_+\to B_+$ be surjective.
The following conditions are equivalent:
\begin{enumerate}
\item[\cond{i}] $T$ is an FM-map.
\item[\cond{ii}] $T$ is additive.
\item[\cond{iii}] $T$ extends to a bounded positive linear surjection $L:A\to B$.
\end{enumerate}
When these conditions hold, $T$ is positively homogeneous, the
extension $L$ is unique, and $L(A_+)=B_+$.
\end{theorem}

\begin{corollary}[Surjective FM-maps on $C^*$-algebras]
\label{cor:CstarFM}
Let $\mathcal A,\mathcal B$ be $C^*$-algebras, not necessarily
unital, and let $T:\mathcal A_+\to\mathcal B_+$ be surjective.
If
\[
 \norm{T(a+b)}=\norm{T(a)+T(b)}\qquad(a,b\in\mathcal A_+),
\]
then $T$ is additive and positively homogeneous, and it extends
uniquely to a bounded positive complex-linear surjection
$\Phi:\mathcal A\to\mathcal B$ with $\Phi(\mathcal A_+)=\mathcal B_+$.
Conversely, the restriction of any such linear map satisfies
the displayed identity.
\end{corollary}
\begin{proof}
The self-adjoint part $\mathcal A_{\mathrm{sa}}$, with product
$x\circ y=(xy+yx)/2$ and the inherited norm, is a real JB-algebra
whose positive cone is $\mathcal A_+$; the same holds for
$\mathcal B$. Theorem~\ref{thm:main} gives a bounded positive
real-linear surjection
$L:\mathcal A_{\mathrm{sa}}\to\mathcal B_{\mathrm{sa}}$ extending $T$.
Define
\[
 \Phi(x+iy)=L(x)+iL(y)
 \qquad(x,y\in\mathcal A_{\mathrm{sa}}).
\]
The unique decomposition into real and imaginary self-adjoint
parts shows that $\Phi$ is a well-defined complex-linear map.
It is positive, and $\norm{\Phi(z)}\leq2\norm L\norm z$.
Surjectivity follows by lifting the real and imaginary parts
of an arbitrary element of $\mathcal B$ separately.
The complex linear span of $\mathcal A_+$ is $\mathcal A$, which
proves uniqueness. The converse follows from additivity.
\end{proof}

Earlier results cover complementary restricted settings. In the
commutative case, Hirota \cite[Theorem~1.1]{HirotaComm} proved
additivity and positive homogeneity for surjective FM-maps between
positive cones of arbitrary commutative $C^*$-algebras, including the
nonunital case. For unital, possibly noncommutative $C^*$-algebras,
Hirota and Oppekepenguin \cite[Theorem~2.1]{HH} established the
weighted Jordan representation for bijective FM-maps on closed
positive cones; their Corollary~2.3 gives the analogous representation
on positive invertible cones. They also treated surjective FM-maps
satisfying the additional condition
$T^{-1}(B_{++})\subset A_{++}$ in \cite[Theorem~2.4]{HH}.

The link between the FM identity and linearity is translation
invariance of the induced pseudometric.
Lemma~\ref{lem:translation} shows that
$d(a,b)=\norm{T(a)-T(b)}$ is translation invariant. For each
$c\in A_+$, the map $F_c(T(a))=T(a+c)-T(c)$ is consequently well
defined even when $T$ is not injective. It is an isometric
self-embedding of the target positive cone at bounded distance
from the identity. Positive-cone rigidity then forces $F_c$ to be
the identity, giving $T(a+c)=T(a)+T(c)$.
The proof uses JB-functional calculus and norming states and
requires neither separability nor an associative operator
representation of the whole algebra.

The order and distance formulas used to establish translation
invariance have the following antecedents. Dong, Li, Moln\'ar,
and Wong \cite[Lemma~2.6]{DLMW} proved the $r=0$ order
characterization that is the prototype of Lemma~\ref{lem:order}
and supplies its basic proof idea. Hatori and Oi
\cite[Lemma~3.1]{HO} established the error-term formulation
$r\geq0$ for arbitrary, possibly nonunital $C^*$-algebras;
their Proposition~3.2 gives the associated distance formula.
Lemma~\ref{lem:order} and Corollary~\ref{cor:distance} extend
these formulations to arbitrary JB-algebras.

The rigidity used here is connected to the stability theory of
positive-cone isometries \cite{DLL,Sun,V}.
Dong, Leung, and Li \cite[Theorem~5.2]{DLL} approximate zero-fixing,
$\delta$-surjective $\varepsilon$-isometries between suitable ordered
positive cones by cone-preserving surjective linear isometries,
with error $2C_F\varepsilon$ independent of $\delta$.
For $C^*$-algebras their \cite[Corollary~5.3(3)]{DLL} gives
$2\varepsilon$, and \cite[Example~5.5]{DLL} establishes the
optimality of the constant $2$.
Section~3 verifies that arbitrary JB-algebras satisfy their
abstract hypotheses with $C_F=1$. Corollary~\ref{thm:rigidity}
is therefore also a consequence of their theorem: bounded
displacement gives approximate surjectivity, and scaling at zero
error identifies the resulting linear isometry with the identity.
For the signed difference maps arising in the approximate FM
argument, Theorem~\ref{thm:quantitative} gives a formulation with
values in the whole algebra and a prescribed asymptotic model
along each positive ray. Its proof is related to the boundary
localization in \cite[Proposition~3.1]{DLL}; we give a
self-contained argument using JB-functional calculus and norming
states. Example~\ref{ex:quantsharp} illustrates the bounds in this
formulation.

The FM identity also admits a quantitative conclusion.
Theorem~\ref{thm:FMstability} shows that a surjective map
$T:A_+\to B_+$ satisfying
\[
 \big|\norm{T(a+b)}-\norm{T(a)+T(b)}\big|\leq\varepsilon
 \qquad(a,b\in A_+)
\]
has a unique bounded positive linear approximation $L:A\to B$
at finite uniform distance, with
$\sup_{a\in A_+}\norm{T(a)-L(a)}\leq3\varepsilon$.
The estimate is an upper bound. The positive image $L(A_+)$ is
dense in $B_+$, but need not equal $B_+$; Example~\ref{ex:stablec0}
exhibits this phenomenon for a continuous bijective map.
Here the hypothesis is an error bound in the FM equation itself,
and the conclusion is positive linear approximation rather than
approximation by a cone-preserving surjective isometry.

A related norm-sum stability result is due to Sun, Cai, and Zheng
\cite[Theorems~3.7--3.8]{SCZ}, who obtain sharp
$3\varepsilon/2$ approximation for surjective, zero-fixing
approximate norm-sum preservers between positive cones of real
continuous-function spaces on compact Hausdorff spaces.
Their hypothesis compares $\norm{S(a)+S(b)}$ with $\norm{a+b}$.
Their Example~3.9 gives the sharpness example for the norm-sum
estimate used in our proof. Our proof first constructs a positive
linear map from the approximate FM condition and then refines
the uniform estimate using norm-sum bounds.

We also obtain consequences for exact norm-sum and mean identities.
Dong, Li, Moln\'ar, and Wong \cite[Theorem~2.5]{DLMW} characterized
surjective maps on positive cones satisfying
$\norm{T(a)+T(b)}=\norm{a+b}$ when at least one of the two
$C^*$-algebras is unital. Hatori and Oi \cite[Theorem~3.3]{HO}
removed this assumption. Corollary~\ref{cor:normsum} extends their
characterization to arbitrary JB-algebras, and
Corollary~\ref{cor:Cstarnormsum} records the $C^*$-algebra case.
Setting $a=b$ in the norm-sum identity gives norm preservation,
which is not required in the FM identity. Thus the FM theorem
also applies to maps outside the norm-sum-preserving class.

Sections~3--5 develop the rigidity principle and its exact and
approximate FM consequences. Section~\ref{sec:Jordan} gives Jordan
representations, including the stability case under an additional
lower distance estimate. The remaining sections treat norm-sum and
mean identities. In particular, Corollary~\ref{cor:Cstarmeans} answers
the norm-midpoint question in the final section of \cite{HH}.

\section{JB-algebra preliminaries and recovery of order and distance}

A JB-algebra is a real Jordan Banach algebra $(A,\circ)$ whose norm
satisfies
\[
 \norm{x\circ y}\leq\norm{x}\norm{y},\qquad
 \norm{x^2}=\norm{x}^2,\qquad
 \norm{x^2}\leq\norm{x^2+y^2}.
\]
Its positive cone $A_+=\{x^2:x\in A\}$ is closed, proper, and
generating, and the norm is monotone on this cone. Continuous
functional calculus gives $x=x_+-x_-$ with
$x_\pm\in A_+$ and $\norm{x_\pm}\leq\norm{x}$.
When $A$ is unital, we denote its positive invertible cone by
\[
 A_{++}=\{a\in A_+:a\text{ is invertible}\}.
\]
We use the standard structure theory in \cite{HS,Shultz}.
The bidual $A^{**}$, with the extended Jordan product, is a unital
JBW-algebra, and the canonical embedding preserves the norm,
product, and order. For a nonunital algebra, expressions involving
$1$ are interpreted in its bidual. The cone $A_+$ is weak-star
dense in $(A^{**})_+$, and the positive normal functionals on
$A^{**}$ are the canonical extensions of the elements of $A^*_+$.

The quadratic representation is
\[
 U_x(y)=2x\circ(x\circ y)-x^2\circ y.
\]
It is positive. For an invertible $x$, it is a linear order
isomorphism with inverse $U_{x^{-1}}$. For positive invertible
$x,y$, one has
\begin{equation}\label{eq:quadraticinverse}
 (U_x y)^{-1}=U_{x^{-1}}(y^{-1}).
\end{equation}
These identities, and the normalization of linear order
isomorphisms by quadratic representations, are reviewed in
\cite[Sections 2.2--2.3]{LRW}.

For $0\leq h\leq1$ in a JBW-algebra, the sequence $h^{1/n}$
increases to its support projection $p=s(h)$ and converges to $p$
in the weak-star topology. If $0\leq u\leq t h$ for some $t>0$,
then $0\leq u\leq tp$ and $U_pu=u$. To justify the latter
assertion, the unital JB-subalgebra generated by $p$ and $u$
is a JC-algebra by the Shirshov--Cohn theorem; see
\cite[Theorem 7.2.5]{HS} and \cite[Section 2.2]{LRW}.
In an operator representation of this subalgebra,
$0\leq u\leq tp$ implies
$(1-p)u(1-p)=0$, hence $u^{1/2}(1-p)=0$ and $u=pup$.
Since $U_pu=pup$ in this representation, the assertion follows.
Thus compression by $U_p$ replaces associative compression
without requiring the whole JBW-algebra to be special.

The zero algebra will be omitted when states or norm-one elements
are used; the conclusions involving it are immediate.

The following lemma extends the order characterization originating
in Dong, Li, Moln\'ar, and Wong \cite[Lemma~2.6]{DLMW}.
Their result treats the case $r=0$ for $C^*$-algebras, and the
proof below follows the basic idea of their argument. Hatori and
Oi \cite[Lemma~3.1]{HO} established the version with the error
term $r\geq0$ for arbitrary $C^*$-algebras, including the nonunital
case. We extend that formulation to arbitrary JB-algebras.

\begin{lemma}\label{lem:order}
Let $C$ be a JB-algebra, $u,v\in C_+$, and $r\geq0$. Then
\begin{equation}\label{eq:order}
 u-v\leq r1
 \quad\Longleftrightarrow\quad
 \norm{u+z}\leq\norm{v+z}+r\quad(z\in C_+).
\end{equation}
\end{lemma}
\begin{proof}
If $u-v\leq r1$, then
$0\leq u+z\leq v+z+r1\leq(\norm{v+z}+r)1$,
which proves the forward implication.
Conversely, the case $u+v=0$ is immediate. Otherwise put
\[
 t=\norm{u+v},\qquad h=(u+v)/t,\qquad p=s(h)\in C^{**}.
\]
Here $s(h)$ denotes the \emph{support projection} of $h$ in the
JBW-algebra $C^{**}$: it is the smallest projection $p$ such
that $p\circ h=h$, equivalently $U_ph=h$.
It is also the spectral projection $\chi_{(0,\infty)}(h)$.
Since $0\leq h\leq1$, spectral calculus and monotone
convergence in a JBW-algebra give
\[
 h^{1/n}\uparrow p,\qquad
 h^{1/n}\longrightarrow p\quad\text{in }\sigma(C^{**},C^*).
\]
Indeed, every positive normal functional $\omega$ satisfies
$\omega(h^{1/n})\uparrow\omega(p)$; the predual is spanned
by its positive functionals. See \cite{HS,Shultz}.
Functional calculus gives
\[
 z_n=th^{1/n}-v\geq th-v=u\geq0,
 \qquad z_n\in C,\qquad\norm{v+z_n}=t.
\]
Hence the assumed inequalities imply
\[
 0\leq u+th^{1/n}-v\leq(t+r)1.
\]
The positive cone of $C^{**}$ is weak-star closed. Thus the
preceding weak-star convergence gives
\[
 0\leq u+tp-v\leq(t+r)1.
\]
\enlargethispage{\baselineskip}
Apply the positive map $U_p$. Since
$U_pu=u$, $U_pv=v$, $U_pp=p$, and $U_p1=p$, we obtain
\[
 u+tp-v\leq(t+r)p.
\]
It follows that $u-v\leq rp\leq r1$.
\end{proof}

In the $C^*$-algebra case, a faithful normal representation of
the bidual also gives $h^{1/n}\to s(h)$ in the strong operator
topology. The weak-star argument above is sufficient for the
order inequality and applies to arbitrary JBW-algebras,
without requiring such an operator representation.

The next corollary extends Hatori and Oi's distance formula
\cite[Proposition 3.2]{HO} from $C^*$-algebras to JB-algebras.

\begin{corollary}\label{cor:distance}
For $u,v\in C_+$, where $C$ is a JB-algebra,
\begin{equation}\label{eq:distance}
 \norm{u-v}=\sup_{z\in C_+}
       \left|\norm{u+z}-\norm{v+z}\right|.
\end{equation}
\end{corollary}
\begin{proof}
The supremum $d$ is at most $\norm{u-v}$ by the reverse triangle
inequality. Lemma~\ref{lem:order} applied in both directions gives
$-d1\leq u-v\leq d1$, hence $\norm{u-v}\leq d$.
\end{proof}

\begin{remark}\label{rem:opentests}
If $C$ is unital, the test elements $z$ in Lemma~\ref{lem:order}
and Corollary~\ref{cor:distance} may be restricted to $C_{++}$,
the positive invertible cone. Indeed, $z+\varepsilon1\in C_{++}$
for $z\in C_+$, and the relevant expressions are norm continuous.
\end{remark}

\section{Quantitative rigidity of positive cones}

We give the signed-valued, raywise formulation needed in the FM
proof, and obtain exact rigidity by setting the error to zero.
The formulation is related to the boundary-localization method of
Dong, Leung, and Li \cite[Proposition~3.1]{DLL}; our proof uses
norming states and asymptotic distances and is self-contained for
arbitrary JB-algebras. We first verify the state-space facts for
nonunital JB-algebras and the applicability of their abstract theorem.
No associative representation of the whole algebra is required.
Throughout this section, $1$ denotes the
unit of $C^{**}$, and positive functionals on $C$ are evaluated there
through their canonical normal extensions.

For a nonzero JB-algebra $C$, set
\[
 Q(C)=\{\omega\in C^*_+:\norm{\omega}\leq1\}.
\]
This is weak-star compact: positivity is given by the weak-star
closed inequalities $\omega(x)\geq0$ for $x\in C_+$, and the
closed dual unit ball is weak-star compact.
If $b\in C_+$ and $\norm b=1$, define
\[
 K_b=\{\omega\in Q(C):\omega(b)=1\}.
\]
We explain why $K_b$ is nonempty even when $C$ is nonunital.
In the unital associative JB-subalgebra generated by $b$ and
$1$ in $C^{**}$, evaluation at the spectral value $1$ defines
a norm-one functional $\rho$ with $\rho(b)=\rho(1)=1$.
The real Hahn--Banach theorem extends $\rho$ to a norm-one
functional $\widetilde\rho$ on $C^{**}$ with
$\widetilde\rho(1)=1$. This extension is positive: if
$0\leq x\leq1$, the order-unit norm gives $\norm{1-x}\leq1$,
so
\[
 \widetilde\rho(x)=1-\widetilde\rho(1-x)\geq0.
\]
Scaling proves positivity on the whole positive cone.
Its restriction $\omega$ to $C$ has norm at most one, and
$\omega(b)=1$ forces $\norm\omega=1$. Thus $\omega\in K_b$.
The functional calculus and order-unit norm used here are
recalled in \cite[Sections 2.1--2.2]{LRW}.
It follows that $K_b$ is nonempty and weak-star compact, and
every element of $K_b$ is a state. In particular, every nonzero
positive element of $C$ admits a norming state.

For completeness, for every $z\in C$ one has
\[
 \max_{\omega\in Q(C)}\omega(z)=\norm{z_+}.
\]
Indeed, the upper bound follows from
$\omega(z)\leq\omega(z_+)\leq\norm{z_+}$.
If $m=\norm{z_+}>0$, put $b=z_+/m$ and choose
$\omega\in K_b$. Functional calculus in $C^{**}$ gives
\[
 0\leq z_-\leq\norm{z_-}(1-b).
\]
The canonical normal extension of $\omega$ satisfies
$\omega(1)=\norm\omega=\omega(b)=1$, so $\omega(z_-)=0$
and $\omega(z)=m$. If $m=0$, the zero functional attains
the upper bound zero. This also explains the use of $Q(C)$
rather than just the state space.

\subsubsection*{Applicability of the Dong--Leung--Li theorem}
Let $K$ be the smallest weak-star closed subset of $Q(C)$ that norms
$C_+$, and put $J=K\cap S_{C^*}^+$, as in \cite[Theorem~2.1]{DLL}.
These sets are distinct from the norming faces $K_b$ above.
We verify that $C$ is a good ordered Banach space in their terminology
and that $J$ norms the whole space with constant $1$.

For positive functionals $\phi,\psi$, their canonical extensions to
$C^{**}$ give
\[
 \norm{\phi+\psi}=\phi(1)+\psi(1)=\norm\phi+\norm\psi.
\]
Consequently every norm-one element of $Q(C)$ is order maximal there.
Every nonzero extreme point of $Q(C)$ has norm one, since an element
of norm strictly between zero and one is a nontrivial convex
combination of zero and its normalization.
By \cite[Theorem~2.1]{DLL}, these extreme points belong to $J$.
For $z\ne0$, choose $\sigma\in\{-1,1\}$ with
$\norm{(\sigma z)_+}=\norm z$. The nonempty weak-star compact face
$\{\omega\in Q(C):\omega(\sigma z)=\norm z\}$ has an extreme point
that is also a nonzero extreme point of $Q(C)$. Hence
\[
 \norm z=\sup_{\omega\in J}|\omega(z)|.
\]
The norm additivity above also proves their condition (HU2).
Condition (HU3) follows from $|z|=(z^2)^{1/2}\geq z,-z$ and
$\norm{|z|}=\norm z$.

For (HU1), fix $b\in C_+$ of norm one and $x\in C$.
For sufficiently large $n$, the norm of $nb+x$ is attained on the
positive side of its spectrum: its negative part has norm at most
$\norm x$, whereas $\norm{nb+x}\geq n-\norm x$.
The preceding extreme-face argument supplies $\omega_n\in J$ with
$\omega_n(nb+x)=\norm{nb+x}$. Then $\omega_n(b)\to1$ and
\[
 \norm{nb+x}-n
 =n(\omega_n(b)-1)+\omega_n(x)\leq\omega_n(x).
\]
Every weak-star cluster point is in $K$ and takes value $1$ at $b$,
so it belongs to $J$. Taking a subsequence realizing the upper limit
and a further weak-star convergent subnet gives the upper bound in
\[
 \lim_{n\to\infty}(\norm{nb+x}-n)
 =\max\{\omega(x):\omega\in J,\ \omega(b)=1\}.
\]
The lower bound follows by evaluating $nb+x$ at any member of the
nonempty compact set $K\cap\{\omega:\omega(b)=1\}\subset J$.
This is (HU1). Finally, $C_+$ is generating and $Q(C)$ norms it.
Thus all the hypotheses in \cite[Theorem~5.2]{DLL} hold for arbitrary
JB-algebras with norming constant $1$; the extra alternative there
is satisfied by $C_F=1$. The zero algebra requires no argument.

\begin{lemma}\label{lem:signedasymptotic}
Let $C$ be a nonzero JB-algebra and let $b\in C_+$ have norm one.
Suppose $x\in C$ and $\gamma\in\mathbb R$ satisfy
$\omega(x)=\gamma$ for every $\omega\in K_b$.
If $w_n\in C$, $r_n=\norm{w_n}\to\infty$, and $w_n/r_n\to b$, then
\[
 \norm{w_n+\tau x}-r_n\longrightarrow\tau\gamma
 \qquad(\tau\in\{-1,1\}).
\]
\end{lemma}
\begin{proof}
Put $\delta_n=\norm{w_n/r_n-b}$. In $C^{**}$ we have
$w_n\geq-\delta_nr_n1$ and
$w_n+\tau x\geq-(\delta_nr_n+\norm x)1$.
For large $n$, $r_n-\norm x>\delta_nr_n+\norm x$.
Thus the norms of both elements are attained on the positive side
of the spectrum. The state-space formula above gives states $\psi_n,\eta_n$ with
\[
 \psi_n(w_n)=r_n,\qquad
 \eta_n(w_n+\tau x)=\norm{w_n+\tau x}.
\]
Consequently
\[
 \tau\psi_n(x)\leq\norm{w_n+\tau x}-r_n
 \leq\tau\eta_n(x).
\]
Moreover $\psi_n(b)\geq1-\delta_n$ and
$\eta_n(b)\geq1-2\norm x/r_n-\delta_n$.
Both values tend to one. Every weak-star cluster point belongs to
$K_b$, so weak-star compactness implies
$\psi_n(x),\eta_n(x)\to\gamma$. Indeed, any subnet on which either scalar convergence failed would
have a further weak-star convergent subnet with limit in $K_b$,
a contradiction. The displayed inequalities prove the claim.
\end{proof}

The next theorem uses a prescribed raywise model and permits signed
values. The boundary test directions used in
\cite[Proposition~3.1, equation~(3.3)]{DLL} can also be adapted to
this situation: norming boundary functionals localize near a fixed
functional, while the raywise assumption makes the norms of the
large signed values attain their positive spectral side.
We give instead a direct proof using the faces $K_b$ and a
one-variable convex function. This formulation will apply to the
signed difference maps arising from the FM identity.

\begin{theorem}[Quantitative rigidity]\label{thm:quantitative}
Let $C$ be a JB-algebra and let $\varepsilon\geq0$.
Suppose $F:C_+\to C$ satisfies $F(0)=0$,
\begin{equation}\label{eq:quantneariso}
 \big|\norm{F(x)-F(y)}-\norm{x-y}\big|\leq\varepsilon
 \qquad(x,y\in C_+),
\end{equation}
and, for every $b\in C_+$,
\begin{equation}\label{eq:quantray}
 \frac{\norm{F(nb)-nb}}n\longrightarrow0.
\end{equation}
Then
\begin{equation}\label{eq:quantbound}
 \sup_{x\in C_+}\norm{F(x)-x}\leq2\varepsilon.
\end{equation}
If also $\norm{F(x)}=\norm x$ for all $x\in C_+$, the bound improves
to $\varepsilon$. Both constants are optimal for the stated classes.
\end{theorem}
\begin{proof}
The zero algebra and $u=0$ require no argument. Fix $u\in C_+$,
$u\ne0$, and put $v=F(u)$. For a sign $\sigma$ with
$s=\norm{(\sigma(v-u))_+}>0$, put $d=\sigma(v-u)$ and
$g(t)=\max_{\omega\in Q(C)}\omega(d+tu)$.
This function is convex, nondecreasing, and $\norm u$-Lipschitz.
A finite convex function on the real line is differentiable except
at at most countably many points. Choose a differentiability point $0<t<s/(2\norm u)$ and set
\[
 m=g(t),\qquad \alpha=g'(t),\qquad b=(d+tu)_+/m.
\]
Then $m\geq s$, $0\leq\alpha\leq\norm u$, and $\norm b=1$.
For $\omega\in K_b$ we have $\omega((d+tu)_+)=m$.
Functional calculus gives
\[
 0\leq(d+tu)_-\leq\norm{(d+tu)_-}(1-b).
\]
Since $\omega(1)=\omega(b)=1$, we get $\omega((d+tu)_-)=0$,
and therefore $\omega(d+tu)=m$.
For both signs of $h$, the inequality
$g(t+h)\geq m+h\omega(u)$ and differentiability give
\[
 \omega(u)=\alpha,\qquad
 \omega(v)=\alpha+\sigma(m-t\alpha)=:\beta
 \quad(\omega\in K_b).
\]
In particular $s\leq|\beta-\alpha|+t\norm u$.
Let $w_n=F(nb)$ and $r_n=\norm{w_n}$.
Equation~\eqref{eq:quantneariso} gives $|r_n-n|\leq\varepsilon$,
and \eqref{eq:quantray} implies $w_n/r_n\to b$.
Lemma~\ref{lem:signedasymptotic} yields
\[
 n-\norm{nb-u}\to\alpha,\qquad
 r_n-\norm{F(nb)-v}\to\beta.
\]
The difference between these two expressions is at most
$2\varepsilon$ in absolute value, so $|\beta-\alpha|\leq2\varepsilon$.
For each fixed $t$ we first take $n\to\infty$; then we let $t\downarrow0$
through differentiability points. This gives $s\leq2\varepsilon$.
Treating both signs, including the trivial case $s=0$, proves
\eqref{eq:quantbound}. If norms are preserved, $r_n=n$ and the same
argument gives $|\beta-\alpha|\leq\varepsilon$.
Optimality follows from Example~\ref{ex:quantsharp} below.
\end{proof}

\begin{corollary}[Exact rigidity]\label{thm:rigidity}
Let $C$ be a JB-algebra. Suppose $F:C_+\to C_+$ is an isometric
embedding, $F(0)=0$, and
\begin{equation}\label{eq:bounded}
 M:=\sup_{x\in C_+}\norm{F(x)-x}<\infty.
\end{equation}
Then $F=\id_{C_+}$, without any surjectivity assumption.
\end{corollary}
\begin{proof}
Bounded displacement implies \eqref{eq:quantray}. Apply
Theorem~\ref{thm:quantitative} with $\varepsilon=0$.
Alternatively, the verification above makes
\cite[Theorem~5.2]{DLL} applicable: $F$ is $M$-surjective, since
$\norm{F(y)-y}\leq M$ for every $y\in C_+$.
At zero error it is the restriction of a cone-preserving surjective
linear isometry $\Phi:C\to C$. For $b\in C_+$,
$n\norm{\Phi(b)-b}=\norm{F(nb)-nb}\leq M$, so $\Phi(b)=b$.
Thus the corollary is also a consequence of the theorem of Dong, Leung, and Li for every JB-algebra.
\end{proof}

The optimal value $2$ already occurs in
\cite[Corollary~5.3(3) and Example~5.5]{DLL}; see also
\cite{SV,Sun,V}. Their three-dimensional maximum-norm example
has bounded displacement and therefore satisfies our raywise
condition. The following two-dimensional example illustrates
sharpness for the stated classes, including the norm-preserving
refinement.

\begin{example}[Sharpness of quantitative rigidity]\label{ex:quantsharp}
Let $C=\mathbb R^2$ with coordinatewise product and the maximum norm,
and fix $\varepsilon>0$. For $a=(s,t)\in C_+$ put
\[
 H(s,t)=(s,\max\{t,\min(s,\varepsilon)\}),\qquad
 F(a)=
 \begin{cases}0,&a=0,\\H(a)+\varepsilon(1,1),&a\ne0.\end{cases}
\]
The map $H$ is nonexpansive, preserves norms, and
$H(a)-a=(0,q(a))$ with $0\leq q(a)\leq\varepsilon$.
Hence
\[
 \norm{a-b}_\infty-\varepsilon
 \leq\norm{H(a)-H(b)}_\infty\leq\norm{a-b}_\infty.
\]
For two nonzero inputs the translation cancels in the difference.
For a nonzero input and zero, the norm increases by exactly
$\varepsilon$. Thus $F$ satisfies \eqref{eq:quantneariso}.
Its displacement is at most $2\varepsilon$, and
\[
 F(\varepsilon,0)=(2\varepsilon,2\varepsilon),\qquad
 \norm{F(\varepsilon,0)-(\varepsilon,0)}_\infty=2\varepsilon.
\]
This proves sharpness of $2$. The map $H$ itself proves sharpness of
$1$ under exact norm preservation. The first example is discontinuous
at zero; it asserts optimality in the class of maps in the theorem,
which has no continuity assumption.
\end{example}

\section{From the FM equation to additivity}

Throughout this section, $T:A_+\to B_+$ is surjective and satisfies \eqref{eq:FM}. Define
\[
 f(a)=\norm{T(a)},\qquad
 d(a,b)=\norm{T(a)-T(b)}\qquad(a,b\in A_+).
\]
The function $d$ is a pseudometric; no injectivity is assumed.

The proofs of $T(0)=0$ and order preservation below follow the
$C^*$-algebra arguments of Hirota and Oppekepenguin
\cite[Lemmas~3.1--3.2]{HH}; we include them in the present
JB-algebra setting.

\begin{lemma}\label{lem:prelim}
We have $T(0)=0$, and $T$ is order preserving. The scalar function $f$ is monotone, subadditive, and positively homogeneous.
\end{lemma}
\begin{proof}
Equation \eqref{eq:FM} at $a=b=0$ gives $T(0)=0$.

Suppose $a\leq b$ and put $c=b-a\in A_+$. For every $x\in A_+$, positivity and \eqref{eq:FM} give
\begin{align*}
 \norm{T(b)+T(x)}
 &=\norm{T(a+c+x)}
 =\norm{T(a+x)+T(c)}\\
 &\geq\norm{T(a+x)}
 =\norm{T(a)+T(x)}.
\end{align*}
Surjectivity and Lemma~\ref{lem:order} with $r=0$ imply $T(a)\leq T(b)$. Hence $f$ is monotone. The triangle inequality yields
\[
 f(a+b)\leq f(a)+f(b),
\]
and \eqref{eq:FM} at $a=b$ gives $f(2a)=2f(a)$.

For an integer $n\geq1$, subadditivity gives $f(na)\leq nf(a)$. Choose $2^k\geq n$. Then
\[
 2^kf(a)=f(2^ka)\leq f(na)+(2^k-n)f(a),
\]
so $f(na)=nf(a)$. It follows that $f(qa)=qf(a)$ for nonnegative rational $q$. For a real $t\geq0$, rational approximation from below and above, together with monotonicity, gives $f(ta)=tf(a)$.
\end{proof}

At this stage only the scalar function $f$ has been shown to be homogeneous. Homogeneity of the map $T$ itself will follow after additivity.

\begin{lemma}\label{lem:translation}
For every $a,b,c\in A_+$,
\begin{equation}\label{eq:translation}
 d(a+c,b+c)=d(a,b).
\end{equation}
\end{lemma}
\begin{proof}
Corollary~\ref{cor:distance}, surjectivity, and \eqref{eq:FM} give
\begin{equation}\label{eq:pullback}
 d(a,b)=\sup_{x\in A_+}|f(a+x)-f(b+x)|.
\end{equation}
Consequently, $d(a+c,b+c)\leq d(a,b)$.

For the opposite inequality, fix $x\in A_+$, and put
\[
 p=a+x,\qquad q=b+x,\qquad D=d(a+c,b+c).
\]
For $n\geq1$, telescope the difference between $f(np+c)$ and $f(nq+c)$ along the $n+1$ elements
\[
 kp+(n-k)q+c,\qquad k=0,\ldots,n.
\]
The two elements in the $k$th consecutive pair can be written as $a+c+z_k$ and $b+c+z_k$, where
\[
 z_k=x+kp+(n-1-k)q\in A_+,\qquad k=0,\ldots,n-1.
\]
Thus \eqref{eq:pullback} gives
\[
 |f(np+c)-f(nq+c)|\leq nD.
\]
Divide by $n$ and use Lemma~\ref{lem:prelim} to obtain
\[
 |f(p+c/n)-f(q+c/n)|\leq D.
\]
Monotonicity and subadditivity imply
\[
 f(p)\leq f(p+c/n)\leq f(p)+f(c)/n,
\]
and the analogous bounds hold for $q$. Letting $n\to\infty$, we find
\[
 |f(a+x)-f(b+x)|\leq D.
\]
Taking the supremum over $x$ proves the reverse inequality.
\end{proof}

\begin{proof}[Proof of Theorem~\ref{thm:main}]
Assume first that $T$ satisfies \eqref{eq:FM}. Fix $c\in A_+$ and define
\begin{equation}\label{eq:Fc}
 F_c(T(a))=T(a+c)-T(c)\qquad(a\in A_+).
\end{equation}
This defines a map on all of $B_+$ by surjectivity. It is well defined: if $T(a)=T(b)$, then $d(a,b)=0$, so Lemma~\ref{lem:translation} gives $T(a+c)=T(b+c)$. Since $T$ is order preserving, its values in \eqref{eq:Fc} lie in $B_+$.

Moreover, $F_c(0)=0$, and
\[
 \norm{F_c(T(a))-F_c(T(b))}
 =d(a+c,b+c)=d(a,b).
\]
Thus $F_c$ is an isometric embedding of $B_+$ into itself. Finally,
\begin{align*}
 \norm{F_c(T(a))-T(a)}
 &\leq\norm{T(a+c)-T(a)}+\norm{T(c)}\\
 &=d(c,0)+f(c)=2f(c).
\end{align*}
Corollary~\ref{thm:rigidity} applies, and $F_c=\id_{B_+}$. Equation \eqref{eq:Fc} now gives
\[
 T(a+c)=T(a)+T(c)\qquad(a,c\in A_+).
\]
This proves additivity.

Next, any additive map $T:A_+\to B_+$ is order preserving and homogeneous for nonnegative rational scalars. For $a\in A_+$, $t\geq0$, and rational $q\leq t\leq r$, positivity gives
\[
 qT(a)\leq T(ta)\leq rT(a).
\]
Taking $q\uparrow t$ and $r\downarrow t$, and using norm closedness of $B_+$, proves positive homogeneity.

Every element of $A$ is a difference of positive elements. Define
\begin{equation}\label{eq:L}
 L(a-b)=T(a)-T(b)\qquad(a,b\in A_+).
\end{equation}
If $a-b=c-d$, then $a+d=c+b$, so additivity shows that the right-hand side is independent of the representation. Additivity and positive homogeneity of $T$ imply real linearity of $L$. Positivity and uniqueness are immediate. Surjectivity follows by writing an arbitrary $y\in B$ as $y=y_+-y_-$ and lifting each positive part separately.

For completeness, boundedness follows directly from positivity.
If $L$ were unbounded on the positive unit ball, then for each
$n\geq1$ we could choose $b_n\in A_+$ such that
\[
 \norm{b_n}\leq1,\qquad \norm{L(b_n)}\geq n2^n.
\]
Putting $a_n=2^{-n}b_n$, we obtain
\[
 \norm{a_n}\leq2^{-n},\qquad \norm{L(a_n)}\geq n.
\]
The norm-convergent sum $a=\sum_{n=1}^{\infty}a_n$ satisfies $a\geq a_n$, and positivity would give $\norm{L(a)}\geq n$ for every $n$, a contradiction. Thus $L$ is bounded on the positive unit ball. The decomposition $x=x_+-x_-$, with $\norm{x_\pm}\leq\norm{x}$, proves boundedness on $A$.

Finally, $L(A_+)=B_+$ because $L|_{A_+}=T$.
Conversely, a linear extension makes $T$ additive, and additivity
implies \eqref{eq:FM}.
\end{proof}

The reduction from the norm-midpoint identity to the FM equation
in \cite[proof of Corollary~2.2]{HH} also applies to surjections
between positive cones of arbitrary JB-algebras. Combining it
with Theorem~\ref{thm:main} gives the following.

\begin{corollary}\label{cor:midpoint}
For a surjection $T:A_+\to B_+$, the conditions in Theorem~\ref{thm:main} are also equivalent to
\[
 \norm{T((a+b)/2)}
 =\norm{(T(a)+T(b))/2}\qquad(a,b\in A_+).
\]
\end{corollary}
\begin{proof}
Only the displayed condition needs consideration. By surjectivity, choose $c\in A_+$ with $T(c)=0$. Apply the condition to $0$ and $2c$ to obtain
\[
 0=\norm{T(c)}=\tfrac12\norm{T(0)+T(2c)}.
\]
Both summands are positive, so $T(0)=0$. Taking $b=0$ gives
$\norm{T(a/2)}=\norm{T(a)}/2$. Hence, for arbitrary $a,b\in A_+$,
\[
 \norm{T(a+b)}
 =2\norm{T((a+b)/2)}
 =\norm{T(a)+T(b)}.
\]
Thus the displayed condition implies \eqref{eq:FM}, and
Theorem~\ref{thm:main} applies. The converse follows from linearity.
\end{proof}

\section{Stability of surjective approximate FM-maps}\label{sec:stability}

We now apply quantitative rigidity to an error term in the FM
identity. The use of selections and nearisometries in FM stability
has a precedent in Tabor \cite{Tabor}, whose domain is a group and
whose range is a whole Banach space. Here both are positive cones,
and the approximating linear operator need not be surjective.

\begin{lemma}\label{lem:stablesum}
Let $C$ be a JB-algebra and $p,q\geq0$. Suppose $S:C_+\to C_+$
has bounded displacement from the identity and satisfies
\[
 \big|\norm{S(x)}-\norm x\big|\leq p,\qquad
 \big|\norm{S(x)+S(y)}-\norm{x+y}\big|\leq q
 \quad(x,y\in C_+).
\]
Then $\sup_x\norm{S(x)-x}\leq p+q$. In fact the bound is
$q+\min\{p,q/2\}$. This estimate is independent of the numerical
bound on the displacement.
\end{lemma}
\begin{proof}
For $u\ne0$, $v=S(u)$, use the direction $b$ and state values
$\alpha,\beta$ from the proof of Theorem~\ref{thm:quantitative}.
Put $r_n=\norm{S(nb)}$; then $|r_n-n|\leq p$ and
$S(nb)/r_n\to b$. Lemma~\ref{lem:signedasymptotic}, now for sums,
gives
\[
 \norm{nb+u}-n\to\alpha,\qquad
 \norm{S(nb)+v}-r_n\to\beta.
\]
Thus $|\beta-\alpha|\leq p+q$, and the same limiting argument for
both signs gives $\norm{S(u)-u}\leq p+q$.
For $u=0$ use $\norm{S(0)}\leq p$.
Taking $x=y$ in the norm-sum assumption bounds the radial error by
$q/2$, proving the refinement.
\end{proof}

In particular, the bound is $3q/2$ when $p=q/2$.
Its sharpness in that case is already witnessed by
\cite[Example~3.9]{SCZ}: their map on $(\ell_\infty^2)_+$ has bounded
displacement, radial error at most $q/2$, norm-sum error at most $q$,
and sends $(q,0)$ to $(q,3q/2)$. This lemma allows maps that are
neither surjective nor zero-fixing, as needed below.

\begin{theorem}[Stability of surjective approximate FM-maps]
\label{thm:FMstability}
Let $A,B$ be JB-algebras, $\varepsilon\geq0$, and let
$T:A_+\to B_+$ be surjective. Assume
\begin{equation}\label{eq:approxFM}
 \big|\norm{T(a+b)}-\norm{T(a)+T(b)}\big|\leq\varepsilon
 \qquad(a,b\in A_+).
\end{equation}
There is a unique real-linear map $L:A\to B$ at finite uniform
distance from $T$ on $A_+$. It is bounded and positive, and
\begin{equation}\label{eq:FMstablebound}
 \sup_{a\in A_+}\norm{T(a)-L(a)}\leq3\varepsilon,
 \qquad \overline{L(A_+)}=B_+.
\end{equation}
In particular $L(A)$ is dense in $B$ and
$\norm{T(a+b)-T(a)-T(b)}\leq9\varepsilon$.
No injectivity, continuity, or normalization at zero is assumed.
The constant $3$ is an upper bound; its optimality is not asserted.
\end{theorem}
\begin{proof}
If $B=\{0\}$ the conclusion is immediate.

\emph{Scalar and distance estimates.}
Put $f(a)=\norm{T(a)}$.
Equation~\eqref{eq:approxFM} and positivity give
\[
 f(0)\leq\varepsilon,\quad |f(2a)-2f(a)|\leq\varepsilon,
 \quad f(a+b)\leq f(a)+f(b)+\varepsilon,
 \quad a\leq b\Rightarrow f(a)\leq f(b)+\varepsilon.
\]
Successive differences of $2^{-n}f(2^na)$ are at most
$2^{-n-1}\varepsilon$, so
\begin{equation}\label{eq:stableg}
 g(a)=\lim_{n\to\infty}2^{-n}f(2^na),\qquad
 |f(a)-g(a)|\leq\varepsilon.
\end{equation}
The function $g$ is monotone and subadditive, with $g(0)=0$ and
$g(2a)=2g(a)$. The integer comparison and rational approximation
in Lemma~\ref{lem:prelim} therefore give $g(ta)=tg(a)$ for $t\geq0$.

Set
\[
\begin{aligned}
 d(a,b)&=\norm{T(a)-T(b)},\\
 D_f(a,b)&=\sup_{x\in A_+}|f(a+x)-f(b+x)|,\\
 h(a,b)&=\sup_{x\in A_+}|g(a+x)-g(b+x)|.
\end{aligned}
\]
Corollary~\ref{cor:distance}, surjectivity and
\eqref{eq:approxFM} imply $|d-D_f|\leq2\varepsilon$.
Thus $D_f$ is finite, and \eqref{eq:stableg} implies finiteness of $h$.
We need the more precise comparison $h\leq D_f\leq h+2\varepsilon$.
The second inequality is immediate. For the first, fix $x\geq0$,
put $p=a+x$, $q=b+x$, and telescope from $nq$ to $np$ along
$kp+(n-k)q$. Each adjacent pair is $a+z_k,b+z_k$, where
$z_k=x+kp+(n-1-k)q\geq0$, $0\leq k<n$.
Hence $|f(np)-f(nq)|\leq nD_f(a,b)$.
Divide by $n$ and use \eqref{eq:stableg} and homogeneity of $g$.
Taking the supremum proves $h\leq D_f$, and hence
\begin{equation}\label{eq:stablehd}
 h(a,b)-2\varepsilon\leq d(a,b)\leq h(a,b)+4\varepsilon.
\end{equation}
The function $h$ is a pseudometric, $h(a,0)=g(a)$, and
$h(a+c,b+c)=h(a,b)$. For the last equality, one inequality follows
by restricting the supremum; the reverse follows by the same
telescoping argument with $c$ added, using
$g(p)\leq g(p+c/n)\leq g(p)+g(c)/n$, as in
Lemma~\ref{lem:translation}. Thus
\begin{equation}\label{eq:stabletranslation}
 |d(a+c,b+c)-d(a,b)|\leq6\varepsilon.
\end{equation}

\emph{Construction of the linear approximation.}
Fix $c\in A_+$ and choose a right inverse $s:B_+\to A_+$ of $T$.
Define $G_c:B_+\to B$ by
\[
 G_c(y)=T(s(y)+c)-T(s(0)+c).
\]
It fixes zero and is a $6\varepsilon$-isometry by
\eqref{eq:stabletranslation}. Moreover \eqref{eq:stablehd} gives
\[
 \norm{T(s(y)+c)-T(s(y))}\leq g(c)+4\varepsilon.
\]
Since $f(s(0))=0$, \eqref{eq:approxFM} also gives
$\norm{T(s(0)+c)}\leq f(c)+\varepsilon$.
Therefore $\norm{G_c(y)-y}\leq g(c)+f(c)+5\varepsilon$, uniformly in
$y$. Theorem~\ref{thm:quantitative} applies, with value space $B$,
and yields $\norm{G_c(y)-y}\leq12\varepsilon$.
Furthermore $g(s(0))\leq\varepsilon$, and \eqref{eq:stablehd} gives
\[
 \norm{T(s(0)+c)-T(c)}\leq5\varepsilon.
\]
For each fixed $a$ choose $s$ so that $s(T(a))=a$. This is also
possible when $T(a)=0$, in which case we take $s(0)=a$.
The estimates do not depend on the selection. Combining them gives
\begin{equation}\label{eq:stable17}
 \norm{T(a+c)-T(a)-T(c)}\leq17\varepsilon
 \qquad(a,c\in A_+).
\end{equation}

The dyadic construction now gives
\[
 L_+(a)=\lim_{n\to\infty}2^{-n}T(2^na)\in B_+,
 \qquad \norm{T(a)-L_+(a)}\leq17\varepsilon.
\]
Scaling \eqref{eq:stable17} proves additivity of $L_+$.
Positivity and rational order bounds give positive homogeneity.
Thus $L(a-b)=L_+(a)-L_+(b)$ defines a positive real-linear map.
The positive-series argument in the proof of
Theorem~\ref{thm:main} proves boundedness without assuming continuity
in advance. Taking norms in the dyadic limit shows
$\norm{L(a)}=g(a)$ for $a\in A_+$.

\emph{The final bound and uniqueness.}
To improve the bound, choose any right inverse $s$ and put
$S(y)=L(s(y))$. Its displacement is at most $17\varepsilon$.
Equations~\eqref{eq:stableg} and \eqref{eq:approxFM} imply
\[
 \big|\norm{S(y)}-\norm y\big|\leq\varepsilon,
 \qquad
 \big|\norm{S(y)+S(z)}-\norm{y+z}\big|\leq2\varepsilon.
\]
Indeed the second estimate follows by inserting $f(s(y)+s(z))$
between $g(s(y)+s(z))$ and $\norm{T(s(y))+T(s(z))}$.
Lemma~\ref{lem:stablesum} with $p=\varepsilon,q=2\varepsilon$ gives
$\norm{S(y)-y}\leq3\varepsilon$.
For each $a$, choose $s(T(a))=a$ once more to obtain
\eqref{eq:FMstablebound}'s first assertion. The map $L$, being
defined by the dyadic limit, is independent of all these choices.

If a real-linear map $M$ is also at finite uniform distance from
$T$ on $A_+$, then $\norm{(L-M)(na)}$ is bounded independently of
$n$. Dividing by $n$ gives $L(a)=M(a)$, and the generating property
of $A_+$ gives $L=M$.
For $b\in B_+$ choose $a_n\in A_+$ with $T(a_n)=nb$.
Then $\norm{L(a_n/n)-b}\leq3\varepsilon/n$, proving density of the
positive image. Density of $L(A)$ follows by taking differences.
The final additive estimate follows from the three approximation
errors for $a+b,a,b$.
\end{proof}

\noindent\emph{Relation to the theorem of Dong, Leung, and Li.}
The existence of a positive linear approximation can alternatively
be obtained from the theorem of Dong, Leung, and Li
\cite[Theorem~5.2]{DLL}, using the JB-algebra
verification in Section~3. Indeed, the scalar and distance estimates
above give $T(a)-T(b)\leq3\varepsilon1$ in $B^{**}$ when $a\leq b$.
For a right inverse $s$ and fixed $c\in A_+$, put
$D_c(y)=T(s(y)+c)-T(c)$. This is a $6\varepsilon$-isometry
with bounded displacement. Since $s(y)+c\geq c$, the approximate
order estimate gives $D_c(y)\geq-3\varepsilon1$, and the distance
estimate above gives $\norm{D_c(0)}\leq5\varepsilon$.
Define
\[
 H_c(0)=0,\qquad H_c(y)=(D_c(y))_+\quad(y\ne0).
\]
The negative part of $D_c(y)$ has norm at most $3\varepsilon$,
so $\sup_{y\in B_+}\norm{H_c(y)-D_c(y)}\leq5\varepsilon$.
Consequently $H_c$ is a zero-fixing
$(6\varepsilon+2\cdot5\varepsilon)=16\varepsilon$-isometry of
$B_+$ with bounded displacement, hence finite approximate surjectivity.
Apply their theorem with $E=F=B$ and $C_E=C_F=1$ to obtain a
cone-preserving surjective linear isometry within $32\varepsilon$
of $H_c$. Bounded displacement forces that isometry to be the
identity. Thus $D_c$ is within
$32\varepsilon+5\varepsilon=37\varepsilon$ of the identity,
which yields a uniform additive bound and then a
positive linear approximation by the same dyadic construction.
Lemma~\ref{lem:stablesum} improves this finite approximation to
$3\varepsilon$, independently of its initial bound. In contrast,
the proof of Theorem~\ref{thm:FMstability} given above applies our
signed-valued Theorem~\ref{thm:quantitative} directly and obtains
the smaller intermediate bound
$17\varepsilon=12\varepsilon+5\varepsilon$.

For $\varepsilon=0$, the theorem recovers the implication from the
FM identity to the positive linear extension in
Theorem~\ref{thm:main}. For positive error the image need not be
closed, even when the original map is a continuous bijection.

\begin{example}[Loss of surjectivity in the linear approximation]
\label{ex:stablec0}
Let $A=B=c_0(\mathbb N,\mathbb R)$ and $\varepsilon>0$. Put
\[
 T(x)_n=x_n/n+\min\{x_n,\varepsilon\}\quad(x\in A_+),
 \qquad L(x)_n=x_n/n\quad(x\in A).
\]
The scalar coordinate maps are strictly increasing. Their inverses
are
\[
 x_n=
 \begin{cases}
 y_n/(1+1/n),&y_n\leq(1+1/n)\varepsilon,\\
 n(y_n-\varepsilon),&y_n>(1+1/n)\varepsilon.
 \end{cases}
\]
For $y\in c_0^+$ the first case holds eventually, so the inverse
sequence belongs to $c_0^+$. Thus $T$ is a bijection; it fixes zero
and is $2$-Lipschitz. For $s,t\geq0$,
\[
 0\leq\min(s,\varepsilon)+\min(t,\varepsilon)
       -\min(s+t,\varepsilon)\leq\varepsilon.
\]
It follows that $\norm{T(x)+T(y)-T(x+y)}\leq\varepsilon$, and hence
\eqref{eq:approxFM} holds. Taking $x=y=\varepsilon e_n$ shows that
the FM error attains $\varepsilon$.
The dyadic limit is $L$, and
$\sup_{x\in A_+}\norm{T(x)-L(x)}=\varepsilon$.
Its range contains all finite sequences and is dense, but
$(1/n)_n$ is not in $L(A)$: its preimage would be the constant
sequence one. Thus $L$ is not surjective.
By uniqueness in Theorem~\ref{thm:FMstability}, no surjective linear
map can replace $L$ at finite uniform distance from $T$.
The example gives the lower bound $1$ for a universal FM approximation
constant; together with the theorem the known bounds here are
$1\leq K_{\mathrm{FM}}\leq3$.
\end{example}

\section{Jordan representations}\label{sec:Jordan}

\begin{proposition}\label{prop:Jordan}
Let $A,B$ be nonzero JB-algebras and let $T:A_+\to B_+$ be a
bijective FM-map. Its extension $L:A\to B$ from
Theorem~\ref{thm:main} is a bounded linear order isomorphism with
bounded positive inverse. Put $h=L^{**}(1_{A^{**}})$. Then
$h\in(B^{**})_{++}$, and there is a unique normal Jordan
isomorphism $\widehat J:A^{**}\to B^{**}$ such that
\begin{equation}\label{eq:Jordan}
 L^{**}=U_{h^{1/2}}\widehat J.
\end{equation}
In particular, $T(a)=U_{h^{1/2}}\widehat J(a)$ for $a\in A_+$.
\end{proposition}
\begin{proof}
If $L(a-b)=0$, then injectivity of $T$ gives $a=b$.
If $L(a-b)\geq0$, choose $c\in A_+$ with $T(c)=L(a-b)$.
Additivity yields $T(a)=T(b+c)$, whence $a=b+c$.
Thus $L$ is bijective and its inverse is positive; the inverse
is bounded by the open mapping theorem.

Let $P=L^{**}$ and $Q=(L^{-1})^{**}$. The bidual maps are positive,
mutually inverse, and weak-star continuous. Positivity follows
from the weak-star density of the positive cones in their bidual
cones. Since $0\leq Q(1)\leq\norm Q\,1$, application of $P$
gives $1\leq\norm Q\,h$. Hence $h$ is positive invertible.
The map
\[
 \widehat J=U_{h^{-1/2}}P
\]
is a unital linear order isomorphism. A unital linear order
isomorphism between JB-algebras is a Jordan isomorphism;
see \cite[Corollary 2.2 and Proposition 2.3]{LRW} and
\cite{WY}. For clarity, unitality and preservation of order in
both directions first imply isometry by the order-unit norm;
the unital surjective isometry theorem for JB-algebras then
gives preservation of the Jordan product.
Multiplication in a JBW-algebra is separately weak-star
continuous, so $U_{h^{-1/2}}$ and $\widehat J$ are weak-star
continuous. The displayed formula determines $\widehat J$
uniquely.
\end{proof}

The same representation applies to the linear approximation in
Theorem~\ref{thm:FMstability} when a lower distance estimate prevents
collapse of the image.

\begin{corollary}\label{cor:stableJordan}
Let $A,B$ be nonzero. Under the hypotheses of
Theorem~\ref{thm:FMstability}, suppose in
addition that some $m>0$ and $\kappa\geq0$ satisfy
\[
 \norm{T(a)-T(b)}\geq m\norm{a-b}-\kappa
 \qquad(a,b\in A_+).
\]
Then $L$ is a bounded linear order isomorphism, and
\[
 L^{**}=U_{h^{1/2}}\widehat J,\qquad h=L^{**}(1),
\]
where $h$ is positive invertible in $B^{**}$ and
$\widehat J:A^{**}\to B^{**}$ is a unital Jordan isomorphism.
\end{corollary}
\begin{proof}
Apply the lower estimate to $na,nb$, use
\eqref{eq:FMstablebound}, and divide by $n$.
This gives
\[
 \norm{L(a)-L(b)}\geq m\norm{a-b}\quad(a,b\in A_+),
\]
and hence $\norm{L(x)}\geq m\norm x$ for $x\in A$.
Both $L(A)$ and $L(A_+)$ are closed: a convergent image sequence
has a Cauchy preimage, and the positive cone is closed.
Density therefore gives $L(A)=B$ and $L(A_+)=B_+$.
The bidual normalization for a bounded linear order isomorphism
in Proposition~\ref{prop:Jordan} gives the stated formula.
\end{proof}

\begin{remark}\label{rem:restriction}
The extension $L$ maps $A$ onto $B$, but its normalization has
the precise image
\[
 \widehat J(A)=U_{h^{-1/2}}(B).
\]
This image must not be identified with $B$ without further
justification. If $A$ and $B$ are unital, however,
$h=L(1_A)=T(1_A)\in B_{++}$ and $U_{h^{\pm1/2}}$ map $B$
onto itself. Thus $J=\widehat J|_A:A\to B$ is a Jordan
isomorphism and
\[
 T(a)=U_{T(1_A)^{1/2}}J(a)\qquad(a\in A_+).
\]
\end{remark}

\begin{example}[Surjectivity does not imply injectivity]
Let $C=C_0((0,1),\mathbb R)$, $A=C\oplus C$, and $B=C$, with
pointwise Jordan products and the usual supremum norms.
The map $T(f,g)=f+g$ on $A_+$ is surjective, additive, and
positively homogeneous, and hence is an FM-map. It is not
injective. Its linear extension is not a Jordan homomorphism:
for nonzero $f\in C_+$, the elements $x=(f,0)$ and $y=(0,f)$
satisfy $x\circ y=0$, whereas $L(x)\circ L(y)=f^2\ne0$.
This example also satisfies $T^{-1}(0)=\{0\}$ on $A_+$.
\end{example}

\begin{example}[An isometric FM-map which is not additive]
\label{ex:nonsurjective}
Inspired by the dominated-coordinate construction in
\cite[Example 1.4]{DLMW}, take the nonunital JB-algebras
$A=B=c_0(\mathbb N,\mathbb R)$ and put
\[
 h(t)=\frac{t^2}{1+t}\quad(t\geq0),\qquad
 T(a)=(a_1,h(a_1),a_2,h(a_2),\ldots)\quad(a\in A_+).
\]
Since $0\leq h(t)\leq t$, this defines a map into $B_+$ and
$\norm{T(a)}=\norm a$. For $a,b\in A_+$,
\begin{align*}
 \norm{T(a)+T(b)}
 &=\sup_n\max\{a_n+b_n,h(a_n)+h(b_n)\}\\
 &=\sup_n(a_n+b_n)=\norm{a+b}=\norm{T(a+b)}.
\end{align*}
Nevertheless,
\[
 T(e_1)=(1,1/2,0,\ldots),\qquad
 T(2e_1)=(2,4/3,0,\ldots)\ne2T(e_1),
\]
so $T$ is neither additive nor positively homogeneous.
It is injective because the odd coordinates recover the input,
and it is not surjective since $e_2\notin T(A_+)$.
Moreover, $h'(t)=1-(1+t)^{-2}\in[0,1)$, so
\[
 \norm{T(a)-T(b)}
 =\sup_n\max\{|a_n-b_n|,|h(a_n)-h(b_n)|\}
 =\norm{a-b}.
\]
Thus even an isometric embedding satisfying the FM equation
need not be additive. Surjectivity in Theorem~\ref{thm:main}
cannot be replaced by injectivity or continuity.
\end{example}

\section{Maps preserving the norm of positive sums}

Hatori and Oi \cite[Theorem 3.3]{HO} proved the
$C^*$-algebra norm-sum characterization using the results of
Section~3 of their preprint. The next corollary extends that
characterization to JB-algebras. Its proof combines
Theorem~\ref{thm:main} with the JB-algebra extensions of their
order and distance formulas.

\begin{corollary}\label{cor:normsum}
Let $A,B$ be JB-algebras. A surjection $\phi:A_+\to B_+$ satisfies
\begin{equation}\label{eq:normsum}
 \norm{\phi(a)+\phi(b)}=\norm{a+b}\qquad(a,b\in A_+)
\end{equation}
if and only if it is the restriction of a Jordan isomorphism
$J:A\to B$. Such an extension is unique.
\end{corollary}
\begin{proof}
Putting $a=b$ gives $\norm{\phi(a)}=\norm a$. Thus $\phi$ is
an FM-map, and Theorem~\ref{thm:main} supplies a bounded linear
surjection $L:A\to B$ extending it. By
Corollary~\ref{cor:distance} and surjectivity,
\begin{align*}
 \norm{\phi(a)-\phi(b)}
 &=\sup_{z\in B_+}|\norm{\phi(a)+z}-\norm{\phi(b)+z}|\\
 &=\sup_{c\in A_+}|\norm{a+c}-\norm{b+c}|=\norm{a-b}.
\end{align*}
Consequently $\phi$ is injective. For every $x\in A$, choose
$a,b\in A_+$ with $x=a-b$. Then
\[
 \norm{Lx}=\norm{\phi(a)-\phi(b)}=\norm{a-b}=\norm x.
\]
Thus $L$ is a surjective linear isometry.
Proposition~\ref{prop:Jordan} shows that $L$ is also an
order isomorphism.

The map $L^{**}$ is a surjective linear isometric order
isomorphism and maps the positive closed unit ball onto the
positive closed unit ball. The greatest elements of these balls
are $1_{A^{**}}$ and $1_{B^{**}}$, respectively. Hence
$L^{**}(1)=1$, so the weight in \eqref{eq:Jordan} is $h=1$.
It follows that $L^{**}$ is a Jordan isomorphism.
Its restriction $J=L$ maps $A$ onto $B$ by construction and
preserves the Jordan product because the canonical bidual
embeddings do. Uniqueness follows from $A=A_+-A_+$.
Conversely, a Jordan isomorphism between JB-algebras is
isometric and additive, and hence satisfies \eqref{eq:normsum}.
\end{proof}

The following $C^*$-algebra specialization is the earlier
theorem of Hatori and Oi \cite[Theorem 3.3]{HO}; it is included
to make the connection with the JB-algebra result explicit.

\begin{corollary}[Hatori--Oi: norm-sum preservers on $C^*$-algebras]
\label{cor:Cstarnormsum}
Let $\mathcal A,\mathcal B$ be arbitrary $C^*$-algebras.
A surjection $\phi:\mathcal A_+\to\mathcal B_+$ satisfies
\[
 \norm{\phi(a)+\phi(b)}=\norm{a+b}\qquad(a,b\in\mathcal A_+)
\]
if and only if it is the restriction of a Jordan $*$-isomorphism
$J:\mathcal A\to\mathcal B$. Such an extension is unique.
\end{corollary}
\begin{proof}
Apply Corollary~\ref{cor:normsum} to the real JB-algebras
$\mathcal A_{\mathrm{sa}}$ and $\mathcal B_{\mathrm{sa}}$.
The resulting Jordan isomorphism
$J_0:\mathcal A_{\mathrm{sa}}\to\mathcal B_{\mathrm{sa}}$ has the
complex-linear extension
\[
 J(x+iy)=J_0(x)+iJ_0(y).
\]
It is bijective and preserves the involution. Since $J_0$
preserves the Jordan product on self-adjoint elements, complex
bilinearity shows that $J$ preserves the Jordan product on all
of $\mathcal A$. Thus $J$ is a Jordan $*$-isomorphism.
Uniqueness follows from the complex linear span of $\mathcal A_+$.
Conversely, a Jordan $*$-isomorphism is isometric and additive,
so its restriction satisfies the displayed identity.
\end{proof}

Corollary~\ref{cor:Cstarnormsum} contains the theorem of Dong,
Li, Moln\'ar, and Wong \cite[Theorem 2.5]{DLMW} as a special
case. Their theorem assumes that at least one of the two
$C^*$-algebras is unital. The removal of that assumption,
including the case in which both algebras are nonunital,
was established by Hatori and Oi \cite[Theorem 3.3]{HO}.
Corollary~\ref{cor:normsum} extends this result to JB-algebras.

\section{Arithmetic and harmonic mean identities on open cones}\label{sec:means}

Throughout this section $A,B$ are nonzero unital JB-algebras,
and $A_{++},B_{++}$ denote their positive invertible cones.
We first isolate the metric argument that will be used for
both arithmetic and harmonic means.

\begin{lemma}\label{lem:opencriterion}
Let $R:A_{++}\to B_{++}$ be surjective and order preserving,
and assume that $\norm{R(x)}\to0$ as $x\to0$ in $A_{++}$.
Suppose there is a continuous positively homogeneous function
$f:A_{++}\to\mathbb R$ such that
\begin{equation}\label{eq:openpullback}
 d(a,b):=\norm{R(a)-R(b)}
 =\sup_{x\in A_{++}}|f(a+x)-f(b+x)|.
\end{equation}
Then $R$ is additive and norm continuous. Injectivity is not
required.
\end{lemma}
\begin{proof}
We first establish
\begin{equation}\label{eq:opentranslation}
 d(a+c,b+c)=d(a,b)\qquad(a,b,c\in A_{++}).
\end{equation}
The inequality $\leq$ follows from \eqref{eq:openpullback}.
For the reverse inequality, fix $x\in A_{++}$, put
$p=a+x$, $q=b+x$, and $D=d(a+c,b+c)$, and telescope along
$kp+(n-k)q+c$, $k=0,\ldots,n$.
Each consecutive pair is of the form $a+c+z_k$, $b+c+z_k$,
where $z_k=x+kp+(n-1-k)q\in A_{++}$. Consequently
$|f(np+c)-f(nq+c)|\leq nD$.
After division by $n$, homogeneity and continuity yield
$|f(a+x)-f(b+x)|\leq D$ as $n\to\infty$.
Taking the supremum proves \eqref{eq:opentranslation}.

For fixed $a\in A_{++}$ and sufficiently small $t>0$, we have
$a-2t1_A\in A_{++}$. Translation invariance gives
\[
 \norm{R(a+t1_A)-R(a-t1_A)}=d(3t1_A,t1_A)\longrightarrow0.
\]
If $\norm{x-a}<t$ and $x\in A_{++}$, order preservation
sandwiches $R(x)$ and $R(a)$ between $R(a-t1_A)$ and
$R(a+t1_A)$. The order-unit norm therefore gives
\[
 \norm{R(x)-R(a)}\leq\norm{R(a+t1_A)-R(a-t1_A)}.
\]
This proves continuity. For $a,c\in A_{++}$ it follows that
\begin{equation}\label{eq:openincrement}
 d(a+c,a)=\lim_{t\downarrow0}d(a+c,a+t1_A)
 =\lim_{t\downarrow0}d(c,t1_A)=\norm{R(c)}.
\end{equation}

Fix $c\in A_{++}$ and set
\[
 F_c(R(a))=R(a+c)-R(c)\qquad(a\in A_{++}).
\]
This defines a map $F_c:B_{++}\to B_+$: order preservation
ensures positivity, and \eqref{eq:opentranslation} ensures
well-definedness even if $R$ is not injective. The same identity
shows that $F_c$ is an isometric embedding. It extends uniquely
by density to an isometric embedding
$\overline F_c:B_+\to B_+$. Taking $a=t1_A$, the limit at zero
and continuity at $c$ imply $\overline F_c(0)=0$.
Equation \eqref{eq:openincrement} gives
\[
 \norm{F_c(R(a))-R(a)}\leq2\norm{R(c)}.
\]
The bound extends to $B_+$ by density. Applying
Corollary~\ref{thm:rigidity} gives $\overline F_c=\id_{B_+}$,
and hence $R(a+c)=R(a)+R(c)$.
\end{proof}

\begin{lemma}\label{lem:opencontinuity}
Suppose $T:A_{++}\to B_{++}$ is surjective and satisfies
\begin{equation}\label{eq:openmidpoint}
 \norm{T((a+b)/2)}=\norm{(T(a)+T(b))/2}
 \qquad(a,b\in A_{++}).
\end{equation}
Then $f(a)=\norm{T(a)}$ is continuous and
\begin{equation}\label{eq:origin}
 \lim_{\substack{x\to0\\x\in A_{++}}}\norm{T(x)}=0.
\end{equation}
Moreover, $T$ satisfies the FM equation on $A_{++}$.
\end{lemma}
\begin{proof}
The triangle inequality gives midpoint convexity of $f$.
Fix $a\in A_{++}$ and choose $r>0$ such that $a\pm k\in A_{++}$
whenever $\norm{k}<r$. Positivity and \eqref{eq:openmidpoint}
give $f(a+k)\leq2f(a)$. If $\norm{k}<2^{-n}r$, iteration of
midpoint convexity gives
\[
 f(a+k)\leq(1-2^{-n})f(a)+2^{-n}f(a+2^nk)
 \leq(1+2^{-n})f(a).
\]
Use the same bound for $-k$ and the inequality
$2f(a)\leq f(a+k)+f(a-k)$ to conclude that
$|f(a+k)-f(a)|\leq2^{-n}f(a)$. Thus $f$ is continuous.

For $\varepsilon>0$, choose $c\in A_{++}$ with
$T(c)=(\varepsilon/4)1_B$. For every sufficiently small
$x\in A_{++}$, the element $2c-x$ is positive invertible, and
\[
 \norm{T(x)}\leq\norm{T(x)+T(2c-x)}
 =2\norm{T(c)}=\varepsilon/2.
\]
This proves \eqref{eq:origin}. Finally, apply
\eqref{eq:openmidpoint} to $2a$ and $t1_A$ to obtain
\[
 2f(a+(t/2)1_A)=\norm{T(2a)+T(t1_A)}.
\]
Letting $t\downarrow0$ gives $f(2a)=2f(a)$. Substitution into
\eqref{eq:openmidpoint} now yields the FM equation.
\end{proof}

\begin{proposition}\label{prop:openFM}
Every surjection $T:A_{++}\to B_{++}$ satisfying
\begin{equation}\label{eq:openFM}
 \norm{T(a+b)}=\norm{T(a)+T(b)}\qquad(a,b\in A_{++})
\end{equation}
is additive and positively homogeneous. It extends uniquely
to a bounded positive linear surjection $L:A\to B$.
\end{proposition}
\begin{proof}
Set $f(a)=\norm{T(a)}$. The FM equation gives subadditivity
and $f(2a)=2f(a)$, and hence \eqref{eq:openmidpoint}.
Lemma~\ref{lem:opencontinuity} supplies continuity of $f$ and
the limit \eqref{eq:origin}.
If $b-a\in A_{++}$, positivity gives $f(b)\geq f(a)$.
For $a\leq b$, replace $b$ by $b+\varepsilon1_A$ and use
continuity. Thus $f$ is order preserving.
For any integer $n\geq1$, choose $2^k>n$. Subadditivity yields
\[
 f(na)\leq nf(a),\qquad
 2^kf(a)\leq f(na)+(2^k-n)f(a).
\]
Therefore $f(na)=nf(a)$. Rational scaling and continuity give
$f(ta)=tf(a)$ for all $t>0$.

For $a\leq b$ and $x\in A_{++}$,
\[
 \norm{T(a)+T(x)}=f(a+x)\leq f(b+x)=\norm{T(b)+T(x)}.
\]
Surjectivity and the open-cone order test in
Remark~\ref{rem:opentests} show that $T(a)\leq T(b)$.
The corresponding distance test gives
\[
 \norm{T(a)-T(b)}
 =\sup_{x\in A_{++}}|f(a+x)-f(b+x)|.
\]
Lemma~\ref{lem:opencriterion} now proves additivity.

Additivity gives rational homogeneity of $T$. On each ray,
$T(sa)\leq T(ta)$ for $0<s<t$; rational approximation in the
closed target cone gives $T(ta)=tT(a)$ for all $t>0$.
Define
\[
 L(a-b)=T(a)-T(b)\qquad(a,b\in A_{++}).
\]
Every element of $A$ has such a representation, and additivity
proves independence of the representation and real linearity.
For $a,b\in A_{++}$ and $t>\norm{a-b}$, order preservation
and additivity give
\[
 -tT(1_A)\leq T(a)-T(b)\leq tT(1_A).
\]
Thus $\norm{L(a-b)}\leq\norm{T(1_A)}\norm{a-b}$ and $L$ is
bounded. For $x\in A_+$ and $\varepsilon>0$,
\[
 T(x+\varepsilon1_A)=L(x+\varepsilon1_A)
 =L(x)+\varepsilon L(1_A)\in B_+.
\]
Letting $\varepsilon\downarrow0$ and using norm closedness of
$B_+$ gives $L(x)\in B_+$.
Surjectivity follows by lifting a representation of an arbitrary
element of $B$ as a difference of two elements of $B_{++}$.
Finally, $A=A_{++}-A_{++}$ ensures uniqueness.
\end{proof}

\Needspace{13\baselineskip}
\begin{theorem}\label{thm:openmidpoint}
Let $T:A_{++}\to B_{++}$ be a bijection. The following
conditions are equivalent:
\begin{enumerate}
\item[\cond{i}] $T(a+b)=T(a)+T(b)$ for all $a,b\in A_{++}$.
\item[\cond{ii}] $T((a+b)/2)=(T(a)+T(b))/2$ for all $a,b\in A_{++}$.
\item[\cond{iii}] $\norm{T(a+b)}=\norm{T(a)+T(b)}$ for all $a,b\in A_{++}$.
\item[\cond{iv}] $\norm{T((a+b)/2)}=\norm{(T(a)+T(b))/2}$ for all $a,b\in A_{++}$.
\item[\cond{v}] There is a Jordan isomorphism $J:A\to B$ such that
\begin{equation}\label{eq:openJordan}
 T(a)=U_{T(1_A)^{1/2}}J(a)\qquad(a\in A_{++}).
\end{equation}
\end{enumerate}
The Jordan isomorphism in \cond{v} is unique. No continuity or
homogeneity of $T$ is assumed.
\end{theorem}
\begin{proof}
Lemma~\ref{lem:opencontinuity} and Proposition~\ref{prop:openFM}
give \cond{iv}$\Rightarrow$\cond{iii}$\Rightarrow$\cond{i}.
If \cond{i} holds, the same proposition extends $T$ to a bounded
positive linear surjection $L:A\to B$. Its inverse on the open
cone is additive, since $T^{-1}(T(a)+T(b))=a+b$.
Applying the extension argument to $T^{-1}$ gives a bounded
positive linear inverse to $L$. Thus $L$ is a linear order
isomorphism. By \cite[Proposition 2.3]{LRW},
$L=U_{L(1_A)^{1/2}}J$ for a Jordan isomorphism $J$.
Since $L(1_A)=T(1_A)$, this proves \cond{v}.
The linear formula in \cond{v} implies all four other conditions,
and \cond{ii} implies \cond{iv} by taking norms.
Uniqueness follows by applying $U_{T(1_A)^{-1/2}}$ on the open
cone and using its real linear span $A$.
\end{proof}

\begin{proposition}[The harmonic mean]\label{prop:harmonicFM}
Let $T:A_{++}\to B_{++}$ be a bijection and define
\[
 a!b=\left(\frac{a^{-1}+b^{-1}}2\right)^{-1}.
\]
The condition
\begin{equation}\label{eq:harmonicFM}
 \norm{T(a!b)}=\norm{T(a)!T(b)}\qquad(a,b\in A_{++})
\end{equation}
is equivalent to each of conditions \cond{i}--\cond{v} in
Theorem~\ref{thm:openmidpoint}. In that case $T$ preserves the
harmonic mean itself, before taking norms.
\end{proposition}
\begin{proof}
For an element $x$ of a unital JB-algebra, put
$\ell(x)=\min\sigma(x)$. This is a concave, order-preserving,
positively homogeneous function with
$|\ell(x)-\ell(y)|\leq\norm{x-y}$.
The identity $\norm{x^{-1}}=1/\ell(x)$ for $x\in A_{++}$
shows that \eqref{eq:harmonicFM} is equivalent, for the bijection
\[
 S(x)=T(x^{-1})^{-1},
\]
to the identity
\begin{equation}\label{eq:ellmidpoint}
 \ell(S((x+y)/2))=\ell((S(x)+S(y))/2).
\end{equation}
Set $f(x)=\ell(S(x))>0$. By concavity of $\ell$, the function
$f$ is midpoint concave.

We first prove its continuity. Fix $a\in A_{++}$ and choose
$r>0$ such that $a\pm k\in A_{++}$ for $\norm{k}<r$.
If $\norm{k}<2^{-n}r$, midpoint concavity gives
\[
 f(a+k)\geq(1-2^{-n})f(a)+2^{-n}f(a+2^nk)
 \geq(1-2^{-n})f(a).
\]
The same bound for $-k$ and the inequality
$f(a+k)+f(a-k)\leq2f(a)$ show that
$|f(a+k)-f(a)|\leq2^{-n}f(a)$.
Thus $f$ is continuous and concave for all real convex
coefficients. It is also order preserving: for $a\in A_{++}$
and $k\in A_+$ the function $t\mapsto f(a+tk)$ is positive and
concave on $[0,\infty)$. A negative secant slope would force
negative values for large $t$, so the function is nondecreasing.

For $u,v\in B_{++}$,
\begin{equation}\label{eq:ellorder}
 u\leq v\quad\Longleftrightarrow\quad
 \ell(u+z)\leq\ell(v+z)\quad(z\in B_{++}).
\end{equation}
For the reverse implication take $z=r1_B-u>0$ with
$r>\norm{u}$, to obtain $\ell(v-u)\geq0$.
For $a\leq b$ in $A_{++}$, scalar monotonicity and
\eqref{eq:ellmidpoint} give
\[
 \ell(S(a)+S(x))=2f((a+x)/2)
 \leq2f((b+x)/2)=\ell(S(b)+S(x)).
\]
Surjectivity and \eqref{eq:ellorder} show that $S$ is order
preserving. Given $\varepsilon>0$, choose $c\in A_{++}$ with
$S(c)=\varepsilon1_B$. All sufficiently small $x\in A_{++}$
satisfy $x\leq c$, and hence $S(x)\leq\varepsilon1_B$.
Therefore $\norm{S(x)}\to0$ as $x\to0$ in the open cone.

Apply \eqref{eq:ellmidpoint} to $2a$ and $t1_A$ and let
$t\downarrow0$ to obtain $f(2a)=2f(a)$.
For fixed $a$, extend $g(t)=f(ta)$ by $g(0)=0$. Since
\[
 0<g(t)=\ell(S(ta))\leq\norm{S(ta)}\longrightarrow0
 \qquad(t\downarrow0),
\]
this extension is continuous at zero. It is therefore continuous
and concave on $[0,\infty)$, so $g(t)/t$ is nonincreasing for
$t>0$. Since $g(2t)=2g(t)$, that ratio
is constant: for $0<s<t$, compare it at $s,t,2^ns$, where
$2^ns\geq t$. It follows that $f(ta)=tf(a)$ for $t>0$.

We also have the distance identity
\begin{equation}\label{eq:elldistance}
 \norm{u-v}=\sup_{z\in B_{++}}|\ell(u+z)-\ell(v+z)|.
\end{equation}
The upper bound follows from the Lipschitz property of $\ell$.
Taking $z=r1_B-u$ and $z=r1_B-v$, with $r$ sufficiently large,
gives the lower bound, since
\[
 \max\{|\ell(v-u)|,|\ell(u-v)|\}=\norm{u-v}.
\]
Surjectivity, \eqref{eq:ellmidpoint}, and homogeneity now give
\[
 \norm{S(a)-S(b)}
 =\sup_{x\in A_{++}}|f(a+x)-f(b+x)|.
\]
All hypotheses of Lemma~\ref{lem:opencriterion} have been
verified, so $S$ is additive. Apply
Theorem~\ref{thm:openmidpoint} to write
$S(x)=U_{s^{1/2}}J(x)$, where $s=S(1_A)=T(1_A)^{-1}$
and $J:A\to B$ is a Jordan isomorphism.
Jordan isomorphisms preserve inverses. By
\eqref{eq:quadraticinverse},
\[
 T(a)=S(a^{-1})^{-1}=U_{s^{-1/2}}J(a)
 =U_{T(1_A)^{1/2}}J(a),
\]
which is condition \cond{v}.

Conversely, $J$ preserves the harmonic mean by linearity and
preservation of inverses. For $c,u,v$ positive invertible,
linearity of $U_{c^{-1}}$ and \eqref{eq:quadraticinverse} give
\[
 (U_cu)!(U_cv)
 =\left(U_{c^{-1}}\left(\frac{u^{-1}+v^{-1}}2\right)\right)^{-1}
 =U_c(u!v).
\]
Thus a map in \cond{v} satisfies $T(a!b)=T(a)!T(b)$ and hence
\eqref{eq:harmonicFM}.
\end{proof}

\begin{example}[The geometric-mean identity is weaker]
For positive invertible elements in a JB-algebra, the geometric
mean is
\[
 a\#b=U_{a^{1/2}}\left((U_{a^{-1/2}}b)^{1/2}\right).
\]
It is the unique positive invertible $x$ satisfying
$U_x(a^{-1})=b$; see \cite[Section 2.5]{LRW}.
The fundamental identity $U_{U_cx}=U_cU_xU_c$, together with
\eqref{eq:quadraticinverse}, shows that
$(U_ca)\#(U_cb)=U_c(a\#b)$ for positive invertible $c$.
Jordan isomorphisms preserve the defining formula. Thus
condition \cond{v} in Theorem~\ref{thm:openmidpoint} implies
preservation of the geometric mean itself.

The converse fails already for $A=B=\mathbb R$. The bijection
$T(t)=t^2$ on $(0,\infty)$ satisfies
\[
 T(a\#b)=ab=T(a)\#T(b),
\]
but $T(1+1)=4\ne2=T(1)+T(1)$ and
$T((1+3)/2)=4\ne5=(T(1)+T(3))/2$.
Consequently the geometric-mean norm identity cannot be added
to the equivalent conditions of Theorem~\ref{thm:openmidpoint}.
\end{example}

\Needspace{15\baselineskip}
\begin{corollary}[Arithmetic and harmonic mean identities on $C^*$-algebras]
\label{cor:Cstarmeans}
Let $\mathcal A,\mathcal B$ be nonzero unital $C^*$-algebras,
and let $T:\mathcal A_{++}\to\mathcal B_{++}$ be a bijection.
Write $a!b=((a^{-1}+b^{-1})/2)^{-1}$.
The following conditions are equivalent:
\begin{enumerate}
\item[\cond{i}] $\norm{T(a+b)}=\norm{T(a)+T(b)}$ for all $a,b\in\mathcal A_{++}$.
\item[\cond{ii}] $\norm{T((a+b)/2)}=\norm{(T(a)+T(b))/2}$ for all $a,b\in\mathcal A_{++}$.
\item[\cond{iii}] $\norm{T(a!b)}=\norm{T(a)!T(b)}$ for all $a,b\in\mathcal A_{++}$.
\item[\cond{iv}] There is a Jordan $*$-isomorphism $J:\mathcal A\to\mathcal B$ such that
\[
 T(a)=T(1_{\mathcal A})^{1/2}J(a)T(1_{\mathcal A})^{1/2}
 \qquad(a\in\mathcal A_{++}).
\]
\end{enumerate}
The Jordan $*$-isomorphism in \cond{iv} is unique.
Each condition is also equivalent to additivity of $T$, and
to the exact arithmetic-midpoint identity. When they hold,
$T$ preserves the harmonic mean itself. No continuity or
homogeneity is assumed.
\end{corollary}
\begin{proof}
Apply Theorem~\ref{thm:openmidpoint} and
Proposition~\ref{prop:harmonicFM} to
$A=\mathcal A_{\mathrm{sa}}$ and $B=\mathcal B_{\mathrm{sa}}$.
Their open cones, inverses, and harmonic means agree with
those in the underlying $C^*$-algebras.
The resulting real Jordan isomorphism complexifies to a
Jordan $*$-isomorphism $J:\mathcal A\to\mathcal B$, as in the
proof of Corollary~\ref{cor:Cstarnormsum}.
For self-adjoint $c,x$ in a $C^*$-algebra, $U_c(x)=cxc$;
thus \eqref{eq:openJordan} becomes the formula in \cond{iv}.
The remaining assertions follow from the same JB-algebra
results and the uniqueness of complexification.
\end{proof}

In particular, the equivalence of \cond{ii} and \cond{iv} in
Corollary~\ref{cor:Cstarmeans} answers the norm midpoint
question posed in the final section of \cite{HH}.

\subsection*{Declaration on the Use of Generative AI}
ChatGPT (OpenAI) was used during preparation of this manuscript for language
editing, organization of the presentation, literature-oriented discussion, and
discussion of mathematical arguments. All mathematical statements and proofs
were independently verified by the authors, who take full responsibility for
the content of the manuscript.

\Needspace{24\baselineskip}


\begin{thebibliography}{99}
\bibitem{DLL}
Y.~Dong, D.~H.~Leung, and L.~Li,
\emph{Stability of isometries between the positive cones of ordered
Banach spaces},
Math. Ann. \textbf{389} (2024), 253--280.
\href{https://doi.org/10.1007/s00208-023-02649-z}{doi:10.1007/s00208-023-02649-z}.

\bibitem{DLMW}
Y.~Dong, L.~Li, L.~Moln\'ar, and N.-C.~Wong,
\emph{Transformations preserving the norm of means between positive
cones of general and commutative $C^*$-algebras},
J. Operator Theory \textbf{88} (2022), 365--406.
Preprint: \href{https://arxiv.org/abs/2104.05909}{arXiv:2104.05909}.

\bibitem{FM}
P.~Fischer and Gy.~Musz\'ely,
\emph{On some new generalizations of the functional equation of Cauchy},
Canad. Math. Bull. \textbf{10} (1967), 197--205.

\bibitem{HS}
H.~Hanche-Olsen and E.~St\o rmer,
\emph{Jordan Operator Algebras}, Monographs and Studies in
Mathematics, vol.~21, Pitman, Boston, 1984.
\href{https://hanche.folk.ntnu.no/joa/}{Author-hosted edition}.

\bibitem{HO}
O.~Hatori and S.~Oi,
\emph{Order isomorphisms on positive cones of non-unital $C^*$-algebras},
unpublished manuscript, 2026; consulted September~16, 2026.

\bibitem{HirotaComm}
D.~Hirota,
\emph{The Cauchy equation and norm additive mappings between positive
cones of commutative $C^*$-algebras},
J. Math. Anal. Appl. \textbf{561} (2026), no.~1, Article 130606.
\href{https://doi.org/10.1016/j.jmaa.2026.130606}{doi:10.1016/j.jmaa.2026.130606}.

\bibitem{HH}
D.~Hirota and J.~Oppekepenguin,
\emph{On the Fischer--Musz\'ely equation for the positive cones of
$C^*$-algebras},
\href{https://arxiv.org/abs/2606.28665v3}{arXiv:2606.28665v3}, 2026.

\bibitem{LRW}
B.~Lemmens, M.~Roelands, and M.~Wortel,
\emph{Hilbert and Thompson isometries on cones in JB-algebras},
Math. Z. \textbf{292} (2019), 1511--1547.
\href{https://doi.org/10.1007/s00209-018-2144-8}{doi:10.1007/s00209-018-2144-8}.

\bibitem{SV}
P.~\v{S}emrl and J.~V\"ais\"al\"a,
\emph{Nonsurjective nearisometries of Banach spaces},
J. Funct. Anal. \textbf{198} (2003), 268--278.
\href{https://arxiv.org/abs/math/0112294v2}{Author preprint: arXiv:math/0112294v2}.

\bibitem{Shultz}
F.~W.~Shultz,
\emph{On normed Jordan algebras which are Banach dual spaces},
J. Funct. Anal. \textbf{31} (1979), 360--376.
\href{https://doi.org/10.1016/0022-1236(79)90010-7}{doi:10.1016/0022-1236(79)90010-7}.

\bibitem{Sun}
L.~Sun,
\emph{Hyers--Ulam stability of $\varepsilon$-isometries between the
positive cones of $c_0$},
Results Math. \textbf{77} (2022), Article 37.
\href{https://doi.org/10.1007/s00025-021-01581-5}{doi:10.1007/s00025-021-01581-5}.

\bibitem{SCZ}
L.~Sun, G.~Cai, and B.~Zheng,
\emph{Approximate norm-additive maps on Banach spaces},
Bull. London Math. Soc. \textbf{56} (2024), 2991--3010.
\href{https://doi.org/10.1112/blms.13115}{doi:10.1112/blms.13115}.

\bibitem{Tabor}
J.~Tabor,
\emph{Stability of the Fischer--Musz\'ely functional equation},
Publ. Math. Debrecen \textbf{62} (2003), 205--211.
\href{https://doi.org/10.5486/PMD.2003.2725}{doi:10.5486/PMD.2003.2725}.

\bibitem{V}
I.~A.~Vestfrid,
\emph{$\varepsilon$-isometries between the positive cones of continuous
functions spaces},
Israel J. Math. \textbf{253} (2023), 989--1000.
\href{https://doi.org/10.1007/s11856-022-2393-4}{doi:10.1007/s11856-022-2393-4}.

\bibitem{WY}
J.~D.~M.~Wright and M.~A.~Youngson,
\emph{On isometries of Jordan algebras},
J. London Math. Soc. (2) \textbf{17} (1978), 339--344.

\end{thebibliography}
\end{document}